\documentclass[12pt,reqno,tbtags]{amsart}

\usepackage{amsmath}
\usepackage{amsrefs}
\usepackage{float}
\usepackage{graphicx}
\usepackage{amsthm}
\usepackage{amssymb}
\usepackage{framed}
\usepackage{url}
\usepackage[margin=1in]{geometry}
\usepackage{caption,sansmath}
\DeclareCaptionFont{sansmath}{\sansmath}
\usepackage[norelsize,ruled,vlined]{algorithm2e}
\newtheorem{thm}{\bf Theorem}

\newtheorem{conj}[thm]{\bf Conjecture}
\newtheorem{question}[thm]{\bf Question}
\newtheorem{lem}[thm]{\bf Lemma}

\theoremstyle{remark}

\newcommand{\bb}[1]{\mathbb{#1}}

\newcommand{\be}{{\bar{e}}}
\DeclareMathOperator{\E}{\mathbb{E}}

\title{Triangle-independent sets vs. cuts 
}
\author{
Sergey Norin \and Yue Ru Sun
}

 \address{Department of Mathematics and Statistics, McGill University.}  \thanks{Supported by an NSERC grant 418520.}
 
\begin{document}

\begin{abstract}
A set of edges $T$ in a graph $G$ is \emph{triangle-independent} if 
$T$ contains at most one edge from each triangle in $G$. Let $\alpha_1(G)$ denote the  maximum size of the triangle-independent set in $G$, and let $\tau_B(G)$ denote  minimum size of a set $F \subseteq E(G)$ such that $G \setminus F$ is bipartite.   We prove that $$\alpha_1(G) + \tau_B(G) \leq \frac{|V(G)|^2}{4},$$
verifying a conjecture due to Lehel, and independently Puleo, and a slightly weaker conjecture of Erd\H{o}s, Gallai and Tuza.
Further, we characterize the graphs which attain the equality.
\end{abstract}
\maketitle
\section{Introduction} 
Erd\H{o}s popularized several quantitative versions of the following question: \emph{``How far from a bipartite graph can a triangle-free graph be?''} Perhaps the most natural way to quantify this question is as follows.
Let $\tau_B(G)$ denote the minimum cardinality of $F \subseteq E(G)$ such that $G \setminus F$ is bipartite.  (All the graphs in this paper are finite and simple.) The following tantalizing conjecture of Erd\H{o}s has received considerable attention over the years. 
 
\begin{conj}[Erd\H{o}s~\cite{ErdProb84}]\label{c:bipartite}
	Let $G$ be a triangle-free graph on $n$ vertices then $\tau_B(G) \leq n^2/25.$ 
\end{conj}

The problem considered in this paper is loosely related to Conjecture~\ref{c:bipartite}. A set of edges $T$ in a graph $G$ is \emph{triangle-independent} if $T$ contains at most one edge from each triangle in $G$. Let $\alpha_1(G)$ denote the  maximum size of the triangle-independent set in $G$. Lehel and Puleo have independently conjectured the following.

\begin{conj}[Lehel (see~\cite{ErdProblems90}), Puleo~\cite{Puleo15}]\label{c:main}
	Let $G$ be a  graph on $n$ vertices then \begin{equation}\label{e:conj}
	\alpha_1(G)+\tau_B(G) \leq \frac{n^2}{4}.
	\end{equation}
\end{conj}  

Let $\tau_1(G)$ denote the  minimum size of the set of edges in $G$ containing at least one edge from each triangle. (Note that $\tau_1(G) \leq \tau_B(G)$.) Erd\H{o}s, Gallai and Tuza proposed the following weakening of Conjecture~\ref{c:main}.

\begin{conj}[Erd\H{o}s, Gallai and Tuza~\cite{ErdProblems90,EGTTriangle}]\label{c:mainEGT}
	Let $G$ be a  graph on $n$ vertices then \begin{equation}\label{e:conjweak}
	\alpha_1(G)+\tau_1(G) \leq \frac{n^2}{4}.
	\end{equation}
\end{conj} 

Erd\H{o}s~\cite{ErdProblems90} noted, and it was reiterated in~\cite{EGTTriangle,ErdProblems99}, that Conjecture~\ref{c:main} is likely to be difficult as several families of graphs attain the equality in  (\ref{e:conjweak}). Define the \emph{join} $H_1 \vee H_2 \vee \ldots \vee H_k$ of a collection of vertex-disjoint graphs $H_1, H_2, \ldots, H_k$ as the graph obtained from their union by joining for all $i \neq j$ every vertex of $H_i$ to every vertex of $H_j$ by an edge. Puleo~\cite{Puleo15}  have shown that a join of any collection of complete balanced bipartite graphs attains the equality in (\ref{e:conj}) and (\ref{e:conjweak}). 

In this paper we prove Conjecture~\ref{c:main} and, consequently, Conjecture~\ref{c:mainEGT}, and show that the examples provided by Puleo are the only ones.

\begin{thm}\label{t:main}
	For every graph $G$, 
	\begin{equation}\label{e:main}
\alpha_1(G) + \tau_B(G) \leq \frac{|V(G)|^2}{4}.
\end{equation}
Moreover the equality holds if and only if $G$ is a join of a collection of  complete balanced bipartite graphs.
\end{thm}

While Conjectures~\ref{c:main} and ~\ref{c:mainEGT} have not received a similar amount of attention to 
Conjecture~\ref{c:bipartite}, several partial results towards Conjecture~\ref{c:main}  have been recently attained. Puleo proved that (\ref{e:main}) holds for triangle-free graphs~\cite{PuleoExtremal15}, and that $\alpha_1(G) + \tau_B(G) \leq 5n^2/16$ holds for every graph $G$ on $n$ vertices~\cite{Puleo15}. Even more recently, Xu~\cite{Xu15} improved this bound to  $\alpha_1(G) + \tau_B(G) \leq 4403n^2/15000$.

Our proof of Theorem~\ref{t:main} follows the general outline of many of the arguments in the area, which were formalized and united in a single framework of flag algebras by Razborov~\cite{FlagAlgebras}. However, we avoid using the flag algebra formalism in this paper. Further, unlike many of the recent results using the flag algebra approach our proof is completely computer-free. While we initially used the computers to verify that our approach can potentially yield the proof of (\ref{e:main}), the final proof was derived by hand.

Let us give a rough sketch of our proof and similar proofs in~\cite{EFPS,SudK4,Puleo15,Xu15}. In all these instances, the goal is to provide an upper bound on $\tau_B(G)$ in terms of other parameters of the graph.  This is done by constructing a partition $(A,B)$ of $V(G)$ such that the number of edges with both ends in $A$ or both ends in $B$, denoted by $\be(A,B)$, is appropriately small. (Clearly, $\tau_B(G)$ equals to the minimum of $\be(A,B)$ taken over all partitions.) The partition achieving the desired bound is constructed in each case using a \emph{local randomized procedure.} 

Rather than attempting to define a local randomized procedure formally, let us provide an example. Erd\H{o}s, Faudree, Pach and Spencer in~\cite{EFPS} proved that $\tau_B(G) \leq n^2/18$ for every triangle-free graph on $n$ vertices, a partial result towards Conjecture~\ref{c:bipartite}. In their proof, the authors of~\cite{EFPS} considered two types of partitions:
\begin{itemize}
	\item  $(N(v),V(G)-N(v))$ for some $v \in V(G)$,
	\item $(N(u) \cup X, N(v) \cup Y)$ for some $uv \in E(G)$, where a partition $(X,Y)$ of $V(G)-N(u)-N(v)$ is chosen uniformly at random
\end{itemize}
(Here $N(v)$ denotes the neighborhood of the vertex $v$ in $G$.) They then proceed to bound the expected size of $\be(A,B)$ of such a partition, where the vertex $v$, and respectively, the edge $uv$, are chosen uniformly at random. The advantage of using  the partitions generated using such local randomized rules, that is partitions, where the part of each vertex depends only on its adjacencies to a small number of the randomly chosen ``fixed'' vertices, is that the resulting expected value can be then expressed in terms of the numbers of configurations induced  in $G$ by a small number of vertices. The relationships between these numbers are in turn derived by using the Cauchy-Schwarz inequality. 

We now describe the procedure we use to generate the partition attaining the bound $\be(A,B) \leq \frac{|V(G)|^2}{4}-\alpha_1(G)$. It is similar to the procedure from~\cite{EFPS} outlined above.  We fix a triangle-independent set $S$ with $|S|=\alpha_1(G)$. We proceed to choose an edge  $uv \in S$ uniformly at random, assign the vertices joined to  $u$ by an edge in $S$  to $A$,  assign the vertices joined to  $v$ by an edge in $S$ to $B$, and extend the partition by recursively applying the same procedure to the subgraph induced by vertices not yet assigned to either $A$ or $B$.

The remainder of the paper is occupied by the analysis of this procedure. In Section~\ref{s:prelim} we formalize the setting in which we will perform the calculations, and derive inequalities and identities used in the proof. In Section~\ref{s:inequality} we prove inequality (\ref{e:main}), and in Section~\ref{s:extremal} we characterize the extremal graphs.  

\section{Proof of Theorem~\ref{t:main}} 

\subsection{Preliminaries}\label{s:prelim}
In the proof of Theorem~\ref{t:main}, we assume that a triangle-independent set $S$ in a graph $G$ on $n$ vertices is fixed and show that there exists a bipartition $(A,B)$ of $G$ such that $\be(A,B)+|S| \leq n^2/4$.
Thus the main object under consideration for us is not the graph $G$,  but a pair $(G,S)$. It is convenient for us to define this pair as a triple $G=(V(G),C(G),S(G))$, such that \begin{itemize}
	\item $(V(G),C(G))$ and $(V(G),S(G))$ are graphs,
	\item $C(G) \cap S(G) = \emptyset$, and
	\item if $uv,uw \in S(G)$ for some $u,v,w \in V(G)$ then $vw \not \in C(G) \cup S(G)$.  
\end{itemize}
We call such a triple $G$ a \emph{triangle-free trigraph}, or ocassionally simply a \emph{trigraph}. Given 
$A,B \subseteq V(G)$, with $A \cap B = \emptyset$, let $e(A,B)$ and $s(A,B)$ denote the number of edges in $C(G) \cup S(G)$, and respectively, in $S(G)$,  with one end in $A$ and the other in $B$, and let $\be(A,B)$ as before denote the number of edges in $C(G) \cup S(G)$ which have both ends in $A$, or both ends in $B$. Finally, let $s(A)$ denote the number of edges in $S(G)$ with both ends in $A$. 

Using the above new notation proving (\ref{e:main}) is equivalent to showing that in every triangle-free trigraph $G=(V,C,S)$ there exists a partition $(A,B)$ of $V$ such that $$\be(A,B)+|S| \leq \frac{|V|^2}{4}.$$ 
Let us repeat the description of the procedure we use to obtain the desired partition, which was outlined in the introduction. For $v \in V$ let $N_S(v):=\{w \in V \:|\: uw \in S\}$ denote the neighborhood of $v$ in the graph $(V,S)$. We generate the parts $A$ and $B$ using a randomized algorithm. We start with $A=B=\emptyset$. At each step of the algorithm if there exists an edge $uv \in S$ such that neither of its ends are assigned to either $A$ or $B$ yet, we choose such an edge in $S$ uniformly at random, assign the neighbors of $u$, which has not been assigned to a part yet to $A$, and assign  the unassigned neighbors of $v$  to $B$. If no such edge exists, we assign the remaining vertices to the parts independently uniformly at random. (See Algorithm~\ref{a:cut}.)

\vskip 20pt
\begin{algorithm}[H]

	\SetKwInOut{Input}{input}\SetKwInOut{Output}{output}
 	\Input{A triangle-free trigraph  $G=(V,C,S)$.}
 	\Output{A partition $(A,B)$ of $V$.}
 	
    $A,B \leftarrow \emptyset$\;
    
 	\While{ $\exists u,v \in V - A - B : uv \in S$}
 	{
 		{\bf choose} such a pair $(u,v)$ uniformly at random\;
 		$A \leftarrow  A \cup (N_S(u)-B)$\;
 		$B \leftarrow  B \cup (N_S(v)-A)$\;
 	}
 	\ForEach{$v \in  V - A - B$}
 	{	{\bf set} either $A \leftarrow A \cup\{v\}$, or
 		$B \leftarrow B \cup\{v\}$  independently uniformly at random.  
    }
 \caption{The cut-generating algorithm}\label{a:cut}
\end{algorithm}
\vskip 20pt

Our proof of Theorem~\ref{t:main} consists of showing that this algorithm on average produces a cut which satisfies the theorem requirements. To characterize the extremal configurations in the new setting we need a few more definitions. A \emph{trigraph $(V, \emptyset, S)$} is \emph{balanced complete bipartite} if the graph $(V,S)$ is \emph{balanced complete bipartite}.  Define a \emph{$C$-join $G=H_1 \vee H_2 \vee \ldots \vee H_k$} of a collection of vertex disjoint graphs $H_1, H_2, \ldots, H_k$  to be obtained from their union by adding an edge joining every vertex of  $H_i$ to every vertex of $H_j$ to  $C(G)$ for all $i \neq j$. We are now ready to state our main technical result, which will imply Theorem~\ref{t:main}.

\begin{thm}\label{t:main2}
Let $G=(V,C,S)$ be a  triangle-free trigraph, and let $(A,B)$ be a random partition of $V(G)$ generated by Algorithm~\ref{a:cut}. Then  
	\begin{equation}\label{e:main2}
	\E[\be(A,B)]+ |S| \leq \frac{|V|^2}{4}.
	\end{equation}
	Moreover, if the equality holds then $G$ is a $C$-join of several complete balanced bipartite trigraphs.
\end{thm}

As noted above, one can derive most of the claims of Theorem~\ref{t:main} by applying Theorem~\ref{t:main2} to the trigraph 
$(V(G),E(G)-S,S)$ for any triangle-independent set $S \subseteq E(G)$ with $|S|= \alpha_1(G)$. It remains only to show that if a graph $G$ is a join of a collection of complete bipartite graphs then $G$ satisfies (\ref{e:main}) with equality. More precisely, it remains to show that $\alpha_1(G) + \tau_B(G) \geq \frac{|V(G)|^2}{4}$ holds for such a graph $G$. This was already observed by Puleo~\cite{Puleo15}, but let us repeat the necessary short argument for completeness. Let $G = K_{t_1,t_1} \vee \ldots \vee K_{t_k,t_k}$. Then 
\begin{equation}\label{e:alpha1}
\alpha_1(G) \geq t^2_1+\ldots+t^2_k,
\end{equation}
as the edges of the original bipartite graphs form a triangle independent set. Also,
\begin{equation}\label{e:EG}
 |E(G)|= 2(t_1+\ldots+t_k)^2 -t^2_1-\ldots-t^2_k.
\end{equation} 
 Moreover, \begin{equation}\label{e:tauB}
 \tau_B(G) \geq |E(G)|-\frac{|V(G)|^2}{4},
\end{equation}
for any graph $G$ as the bipartite graph on $|V(G)|$ vertices has at most $\frac{|V(G)|^2}{4}$ edges by Mantel's theorem. Combining (\ref{e:alpha1}), (\ref{e:EG}) and (\ref{e:tauB}) yields the desired inequality.
 
 \vskip 10pt
It remains to prove Theorem~\ref{t:main2} and the rest of the section is occupied by the proof. It proceeds by exploring the relations between the  numbers of several 4-vertex configurations in $G$. In the following definitions, we will use the characteristic functions of $C$ and $S$. We define $c:V^2 \to \{0,1\}$ by $c(uv)=1$ if $uv \in C$, and $c(uv)=0$ otherwise. (For brevity here and below we will write $f(uv)$ rather than $f(u,v)$ for the value of a function $f$ defined on $V^2$.)
Analogously,  define $s:V^2 \to \{0,1\}$ by $s(uv)=1$ if $uv \in S$, and $s(uv)=0$, otherwise. Finally, let  $n:V^2 \to \{0,1\}$, defined by $n(uv)=1-c(uv)-s(uv)$, be the characteristic function of non-edges. (Note that $n(vv)=1$ for every $v \in V(G)$.) 
Let $$P_4(G)=\sum_{(u,v,w,x) \in V^4} s(uv)s(vw)s(wx)(n(xu)+c(xu)),$$
    $$C_4(G)=\sum_{(u,v,w,x) \in V^4} s(uv)s(vw)s(wx)s(xu),$$
    $$K_{1,3}(G)=\sum_{(u,v,w,x) \in V^4} s(uv)s(uw)s(ux),$$ 
    $$D(G)=\sum_{(u,v,w,x) \in V^4} (n(uv)+c(uv))s(uw)s(ux)n(vw)n(vx).$$
    $$R(G)=\sum_{(u,v,w,x) \in V^4}s(uv)s(uw)n(wx)c(vx).$$
The names for the first three of the above quantities reflect the types of objects counted by them, e.g. $P_4(G)$ is proportional to the number of induced subgraphs isomorphic to $P_4$ in the graph $(V,S)$. For convenience of the reader the quadruples counted by each of the above functions are shown in Figure~\ref{f:fig1}. (In the figure, the edges in $S$ are indicated by thick lines, the edges in $C$ by thin lines, and pairs of vertices that can not be adjacent by dashed lines. When several possibilities are allowed for a pair of vertices, parallel edges are drawn.) 
\begin{figure}
	\begin{center}
		\includegraphics[scale=1]{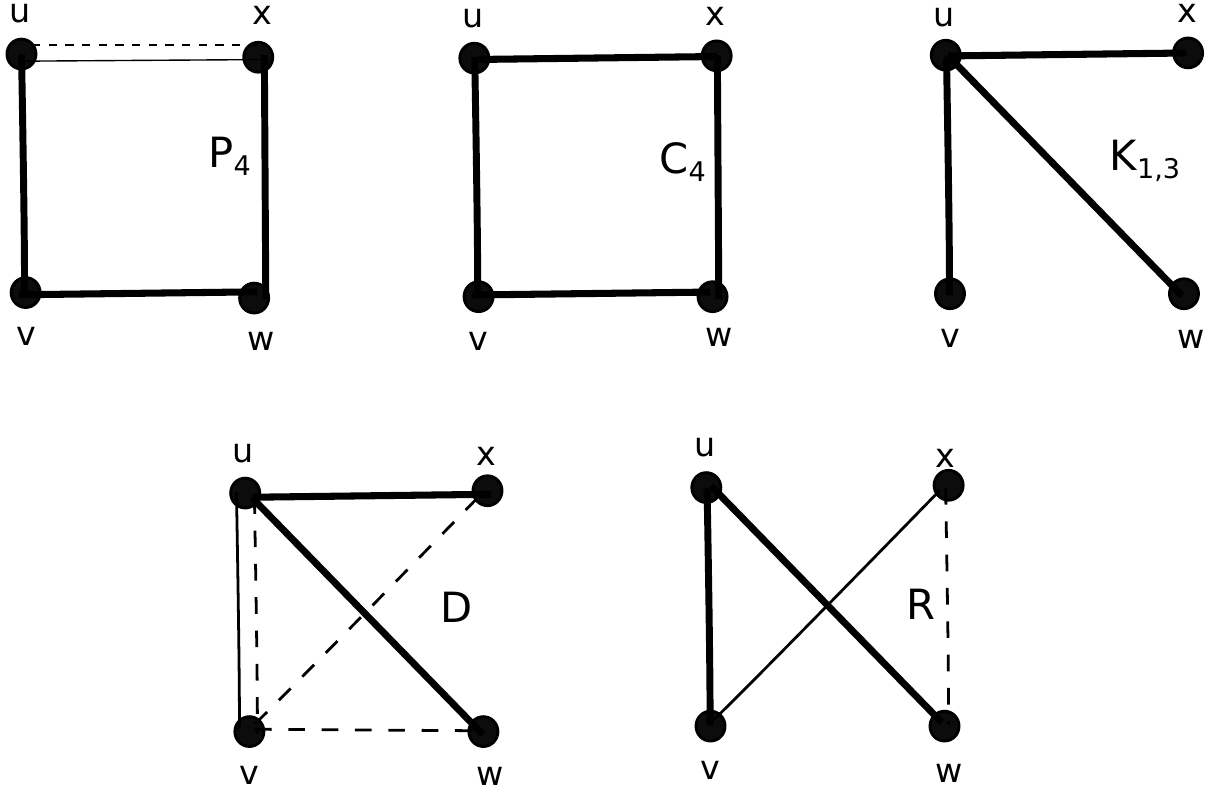}%
	\end{center}
	\caption{Configurations counted in the proof of Theorem~\ref{t:main2}.}
	\label{f:fig1}
\end{figure}

The following technical lemma is proved via a couple of  standard applications of the Cauchy-Schwarz inequality.

\begin{lem}\label{l:CS} Let $G=(V,C,S)$ be a  triangle-free trigraph, then
	\begin{equation}\label{e:CS-s}
	P_4(G)+C_4(G) \leq K_{1,3}(G)
	\end{equation}
	and
	\begin{equation}\label{e:CS-n}
	P_4(G) \leq D(G).
	\end{equation}
\end{lem}	

\begin{proof}
We start by proving (\ref{e:CS-s}). Note that 
$$P_4(G)+C_4(G) = \sum_{(u,v,w,x) \in V^4}s(uv)s(vw)s(wx).$$
We have
\begin{align*}
0 &\leq \sum_{(u,v) \in V^2}\left(s(uv) \sum_{w \in V}(s(uw) -s(vw))^2 \right) \\&=  \sum_{(u,v) \in V^2}\left(s(uv) \sum_{w \in V}(s(uw) -s(vw)) \sum_{x \in V}(s(ux) -s(vx))  \right) \\ &=
\sum_{(u,v,w,x) \in V^4}(s(uv)s(uw)s(ux) +s(vu)s(vw)s(vx)) \\ &-
\sum_{(u,v,w,x) \in V^4}(s(wv)s(vu)s(ux) +s(wu)s(uv)s(vx)) \\ &=
2(K_{1,3}(G) - P_4(G)-C_4(G)),
\end{align*}	
as desired.

The proof of (\ref{e:CS-n}) is similar:
\begin{align*}
0 &\leq \sum_{(u,v) \in V^2}\left((n(uv)+c(uv)) \sum_{w \in V}(s(uw)n(vw) -s(vw)n(uw)) \sum_{x \in V}(s(ux)n(vx) -s(vx)n(ux))  \right) \\ 
&\leq 2D(G) - \sum_{(u,v,w,x) \in V^4}(n(uv)+c(uv))s(uw)n(vw)s(vx)n(ux)s(xw)\\ & \qquad \qquad - \sum_{(u,v,w,x) \in V^4}(n(uv)+c(uv))s(vw)n(uw)s(ux)n(vx)s(xw)  \\&
= 2D(G) - \sum_{(u,v,w,x) \in V^4} (n(uv)+c(uv))s(uw)s(vx)s(xw) \\& \qquad \qquad - \sum_{(u,v,w,x) \in V^4}(n(uv)+c(uv))s(vw)s(ux)s(xw) = 
2D(G) -2 P_4(G),
\end{align*}
as desired. Note that in the inequalities above, we used the fact that $G$ is triangle-free, and thus $s(uv)s(vw)=s(uv)s(vw)n(uw)$ for all $(u,v,w) \in V^3$.
\end{proof}

The proof of (\ref{e:main2}) involves, besides the inequalities of Lemma~\ref{l:CS}, several identities, which require sums similar to those considered in the proof of Lemma~\ref{l:CS}. These identities are typically routine to verify. However, we will derive the most unwieldy of them separately in the next lemma.

\begin{lem}\label{l:bigsum} Let $G=(V,C,S)$ be a  triangle-free trigraph, then
	\begin{align} 
	&\sum_{(u,v,w,x) \in V^4}s(uv)c(wx)(s(uw)+s(vw))(1-s(ux)-s(vx))  \notag\\ &+ \frac 1 2 \sum_{(u,v,w,x) \in V^4}s(uv)(1-s(uw)-s(vw)) (1-s(ux)-s(vx)) \notag\\&= |V|^2|S|-3P_4(G)-C_4(G)-K_{1,3}(G) -2D(G) - 2R(G).\label{e:fbound2}
	\end{align} 
\end{lem}

\begin{proof} We have \begin{align}
	\frac 1 2 &\sum_{(u,v,w,x) \in V^4}s(uv)(1-s(uw)-s(vw)) (1-s(ux)-s(vx)) 
	\notag \\&= \frac12  \sum_{(u,v,w,x) \in V^4}s(uv) - \sum_{(u,v,w,x) \in V^4}s(uv)(s(uw)+s(vw)) \notag \\&+ \frac 1 2 \sum_{(u,v,w,x) \in V^4}s(uv)(s(uw)+s(vw))(s(ux)+s(vx)) \notag \\&=  |S||V|^2  -  \sum_{(u,v,w,x) \in V^4}s(uv)(s(uw)+s(vw)) + P_4(G)+C_4(G)+K_{1,3}(G).
	\end{align}
Thus (\ref{e:fbound2}) can be rewritten as 
 	\begin{align} 
	&\sum_{(u,v,w,x) \in V^4}s(uv)(s(uw)+s(vw)) \notag\\&- \sum_{(u,v,w,x) \in V^4}s(uv)c(wx)(s(uw)+s(vw)) (1-s(ux)-s(vx)) \notag\\&= 4P_4(G)+2C_4(G)+2K_{1,3}(G)+2D(G) + 2\sum_{(u,v,w,x) \in V^4}s(uv)s(uw)n(wx)c(vx).\label{e:fbound3}
	\end{align} 
We have 	\begin{align} 
&\sum_{(u,v,w,x) \in V^4}s(uv)(s(uw)+s(vw)) \notag\\&- \sum_{(u,v,w,x) \in V^4}s(uv)c(wx)(s(uw)+s(vw)) (1-s(ux)-s(vx)) \notag\\&= \sum_{(u,v,w,x) \in V^4}s(uv)(s(uw)+s(vw))(n(wx)+s(wx)+c(wx)s(ux)+c(wx)s(vx))
\label{e:fbound4}
\end{align} 
Further,
\begin{equation}\label{e:aux1}
\sum_{(u,v,w,x) \in V^4}s(uv)(s(uw)+s(vw))s(wx) = 2P_4(G)+2C_4(G),
\end{equation}	
and 	
\begin{align}\label{e:aux2}
&\sum_{(u,v,w,x) \in V^4}s(uv)(s(uw)+s(vw))(n(wx)+c(wx)s(ux)+c(wx)s(vx))\notag \\ &= 2\sum_{(u,v,w,x) \in V^4}s(uv)s(uw)(n(wx)+c(wx)s(vx)) \notag \\
&= 2\sum_{(u,v,w,x) \in V^4}s(uv)s(uw)s(vx)(n(wx)+c(wx))\notag \\&+ 2\sum_{(u,v,w,x) \in V^4}s(uv)s(uw)n(vx)n(wx)\notag\\&+2\sum_{(u,v,w,x) \in V^4}s(uv)s(uw)c(vx)n(wx) \notag\\&= 2P_4(G) + 2(2K_{1,3}(G)+D(G)) +2R(G).
\end{align}
Finally, (\ref{e:fbound3}) follows by substituting (\ref{e:aux1}) and (\ref{e:aux2}) into (\ref{e:fbound4}), as desired. 
\end{proof}

\subsection{Proof of (\ref{e:main2}).}\label{s:inequality}
We prove the inequality (\ref{e:main2}) by induction on $|V|$. The base case $|V|=0$ is trivial.

For the induction step, note that if $|S|= 0$ then $\E[\be(A,B)]= |C|/2 \leq |V|^2/4$, and so we assume that $|S| \neq 0$.
Consider for the moment a single edge $uv \in S$.
Let $\E_{uv}[\be(A,B)]$ denote the expected value of $\be(A,B)$ subject to the edge $uv \in S$ being chosen at the first step of Algorithm~\ref{a:cut}. Let $A'=N_S(v)$, $B'=N_S(u)$, and $Z=V - A' -B'$.
By the induction hypothesis, 
\begin{equation}\label{e:induction}
\E_{uv}[\be(A \cap Z, B\cap Z)] + s(Z) \leq \frac{|Z|^2}{4}.
\end{equation} 
As every edge with one end in $A'\cup B'$ and another in $Z$ is counted in $\be(A \cap Z, B\cap Z)$ with probability $1/2$, we have
\begin{equation}\label{e:cut}
\E_{uv}[\be(A, B)] = \frac{1}{2}e(A'\cup B',Z)+ \E_{uv}[\be(A \cap Z, B\cap Z)].
\end{equation} 
Further, let us note that
\begin{equation}\label{e:scount}
|S|=s(A',B')+s(A'\cup B',Z)+s(Z).
\end{equation}
Combining, (\ref{e:induction}),(\ref{e:cut}), and (\ref{e:scount}) we obtain 
\begin{equation}\label{e:local}
\E_{uv}[\be(A,B)]+ |S| \leq \frac{1}{2}e(A'\cup B',Z)+ \frac 14 |Z|^2 + s(A',B')+s(A'\cup B',Z).
\end{equation}
We now rewrite the right side of (\ref{e:local}) in terms of characteristic functions introduced earlier. Note that for $w \in V$, we have $w \in A' \cup B'$ if and only if $s(uw)+s(vw)=1$. We have
\begin{align}
\frac{1}{2}e(A'\cup B',Z)&+s(A'\cup B',Z) \notag\\ &= \sum_{(w,x) \in V^2}\left(\frac{3}{2}s(wx)+\frac{1}{2}c(wx)\right)(s(uw)+s(vw))(1-s(ux)-s(vx)),\label{e:local1}\\
|Z|^2 &= \sum_{(w,x) \in V^2} (1-s(uw)-s(vw)) (1-s(ux)-s(vx)),\label{e:local2}\\
 s(A',B') &= \sum_{(w,x) \in V^2}s(wx)s(uw)s(vx).\label{e:local3}   
\end{align}
The identities (\ref{e:local1}),(\ref{e:local2}) and (\ref{e:local3}), motivate the definition of the following function
\begin{align*}
f(u,v,w,x)&=s(uv)((3s(wx)+c(wx))(s(uw)+s(vw))(1-s(ux)-s(vx)) \\ &+ \frac 1 2 (1-s(uw)-s(vw)) (1-s(ux)-s(vx)) +2s(wx)s(uw)s(vx)).
\end{align*} for $(u,v,w,x) \in V^4$.
Then (\ref{e:local}) can be rewritten as 
\begin{equation}\label{e:localnew}
\E_{uv}[\be(A,B)]+ |S| \leq \frac{1}{2}\sum_{(w,x) \in V^2}f(u,v,w,x).
\end{equation}
Let $F(G) = \sum_{(u,v,w,x) \in V^4}f(u,v,w,x).$ We estimate $F(G)$ as follows. Note that 
\begin{equation}\label{e:fbound1}
\sum_{(u,v,w,x) \in V^4}s(uv)s(wx)(s(uw)+s(vw))(1-s(ux)-s(vx)) = 2P_4(G).
\end{equation}
Combining the results of  Lemmas~\ref{l:CS} and~\ref{l:bigsum}  and the identity (\ref{e:fbound1}),  we obtain
\begin{align}
F(G) &\leq |V|^2|S| +3P_4(G)+C_4(G)-K_{1,3}(G) -2D(G) \notag\\ &= 
 |V|^2|S| + (P_4(G)+C_4(G)-K_{1,3}(G)) + 2(P_4(G)-D(G)) \leq|V|^2|S|. \label{e:fbound} 
\end{align}
Finally, we have
\begin{align*}
|S|(\E[\be(A,B)]+ |S|) = \sum_{uv \in S}\left(\E_{uv}[\be(A,B)]+ |S|\right) \stackrel{(\ref{e:localnew})}{=}  \frac{1}{4}F(G) \stackrel{(\ref{e:fbound})}{\leq} \frac{1}{4}|V|^2|S|,
\end{align*}
implying (\ref{e:main2}).

\subsection{Characterizing extremal trigraphs}\label{s:extremal}
In this subsection, we prove that if a triangle-free trigraph  $G=(V,C,S)$ satisfies (\ref{e:main2}) with equality then $G$  is a $C$-join of complete balanced bipartite trigraphs. The proof is by induction on $|V|$. 

The base case $|V|=0$ is trivial.
For the induction step,  let $H=(V,S)$ be the graph formed by the edges in $S$. It is easy to see that if $G$ satisfies (\ref{e:main2}) with equality then the subgraph induced in $G$
by vertices of every component of $H$ satisfy (\ref{e:main2}) with equality, and, moreover, $uv \in C$ for every pair of vertices $u,v \in V$ belonging to different components of $H$.
Thus the desired result follows from the induction hypothesis, unless $H$ is connected, and so we assume it is. It remains to prove that $H$ is a complete balanced bipartite graph.
 
Note that it follows from (\ref{e:fbound}) that (\ref{e:CS-s}), (\ref{e:CS-n}) must hold with equality in $G$, and so must the first inequality in 
(\ref{e:fbound}). In particular,  we have that $\deg_{H}(u)=\deg_H(v)$ for every $uv \in S$, as (\ref{e:CS-s}) holds with equality.
Therefore we have
\begin{enumerate}
	\item[(i)] $\deg_{H}(u)=\deg_H(v)$  for all $u,v \in V$,
	\item[(ii)] $(n(uv)+c(uv))s(vw)n(uw)s(ux)n(vx)(n(xw)+c(xw))=0$ 
	for all $u,v,w,x \in V$,
	\item [(iii)] $s(uv)s(uw)n(wx)c(vx)=0$ for all $u,v,w,x \in V$.
\end{enumerate}

 Our next goal is to show that $C = \emptyset$. Suppose for a contradiction that  $C \neq \emptyset$, and choose $u_1,u_2,\ldots,u_n$, such that $u_1u_n \in C$,
	$u_iu_{i+1} \in S$ for $1 \leq i \leq n-1$, and subject to the above $n$ is minimum. Then $n \geq 4$, as $G$ is triangle-free.
	If $n\geq 5$, then $s(u_1u_2)s(u_2u_3)n(u_3u_n)c(u_1u_n)=1$, in contradiction with (iii). Thus $n=4$. Consider now arbitrary $v \in N_S(u_1)$. If $vu_4 \in C$ then   $s(u_1u_2)s(u_1v)n(u_2u_4)c(vu_4)=1$, once again contradicting (iii). Thus $n(vu_4)=1$, implying $v \in N_S(u_3)$, as otherwise $(n(vu_3)+c(vu_3))s(vu_1)n(vu_4)s(u_3u_4)n(u_1u_3)c(u_1u_4)=1$, contradicting (ii). It follows that $N_S(u_1) \subseteq N_S(u_3)- \{u_4\}$, contradicting (i).
	Thus $C=\emptyset$.
 
Note that for every edge $uv \in S$ the trigraph obtained from $G$ by deleting the vertices of $N_S(v) \cup N_S(u)$ must also satisfy (\ref{e:main2}) with equality. As $C= \emptyset$, it follows from the induction hypothesis that $V - N_S(v)-N_S(u)$ induces a complete balanced bipartite graph in $H$ for all $uv \in S$. Thus, if $N_S(v) \cup N_S(u) \neq V$ for some $uv \in S$, then there exist $w,x \in V - N_S(v)-N_S(u)$ such that $wx \in S$. However, in such  a case we have  $s(uv)s(wx)n(uw)n(ux)n(vw)n(vx)=1$, contradicting (ii). Thus $N_S(v) \cup N_S(u) = V$ for all $uv \in S$, and thus $\deg_H(v)=|V|/2$ for every $v \in V$ by (i). It follows that $H$ is a complete balanced bipartite graph, as desired.


\section{Concluding remarks}\label{s:remarks}

\subsection*{Influence of the structure of extremal examples on our proof.}
The range of possible techniques one can use to approach a given extremal problem is frequently dictated by the structure of the extremal examples. Inequalities used throughout the proof must be tight for these examples, and, conversely, the class of all the extremal  examples is often characterized by this property.

As mentioned in the introduction, the authors of~\cite{EGTTriangle} noted that Conjecture~\ref{c:mainEGT} might be difficult due to the large number of extremal examples. However, considered as trigraphs, extremal examples form a simple family from the point of view of \emph{local graph theory} (a term recently coined by Linial, see e.g.~\cite{LinialLocal}). More specifically, a trigraph $G=(V,C,S)$ is a   $C$-join of a collection complete balanced bipartite trigraphs if and only if it satisfies the following concise set of  conditions \begin{itemize}
\item $n(uv)n(vw)(s(uw)+c(uw))=0$ for all $u,v,w \in V,$ 
\item $n(uv)s(vw)c(uw)=0$ for all $u,v,w \in V(G),$
\item $C_4(G)=K_{1,3}(G)$, and
\item $\sum_{v \in V}s(uv)>0$ for every $u \in V(G)$
\end{itemize}
There does not, however, seem to exist a similar description of joins of complete balanced bipartite graphs in terms of relations between numbers of induced subgraphs of bounded size. 

The above observation suggest that any proof of Conjecture~\ref{c:main} using the standard techniques must take the structure of the maximum triangle-independent set into account. Further, the partitions corresponding to maximum cuts in the joins of complete balanced bipartite graphs can not be generated using a single step local randomized algorithm, similar to the procedure from~\cite{EFPS} described in the introduction. Hence recursion appears to be necessary. Algorithm~\ref{a:cut} used in our proof is the simplest procedure we could come up with which passes the above tests.

\subsection*{Algorithmic aspects.}
Algorithm~\ref{a:cut}, given the graph $G$ and a triangle-independent set $S \subseteq E(G)$ with $|S|=\alpha_1(G)$  as an input, provides a randomized way of generating a partition of $V(G)$, which certifies the validity of (\ref{e:main}). It is not hard to modify Algorithm~\ref{a:cut} to obtain a deterministic efficient algorithm with the same properties: Rather than choosing the edge of $uv \in S$ uniformly at random it suffices to choose it so that 
$\sum_{(w,x) \in V^2}f(u,v,w,x) \leq \frac{|V|^2}{2}$. The proof of Theorem~\ref{t:main2} shows that such a choice is possible.

The requirement that Algorithm~\ref{a:cut} has access to the maximum triangle independent set as an input is more substantial. It is easy to see that computing $\alpha_1(G)$ is NP-hard. This is not necessarily an insurmountable obstacle, but we do not know how to efficiently generate a partition $(A,B)$ of $V(G)$, satisfying $\be(A,B) \leq |V|^2/4 - \alpha_1(G)$, given only the graph $G$ as an input.

\subsection*{Making triangle-free graphs bipartite.}
It is tempting to try using Algorithm~\ref{a:cut} to attack Conjecture~\ref{c:bipartite}. However, the straightforward analysis, along the lines of the proof of Theorem~\ref{t:main2}, does not seem to allow to substantially improve on the bound  $\tau_B(G) \leq |V(G)|^2/18$ from~\cite{EFPS}. 

Moreover, it is impossible to prove Conjecture~\ref{c:bipartite} using only Algorithm~\ref{a:cut} to generate the bipartition: when applied to the Clebsch graph, the algorithm always outputs a bipartition $(A,B)$ such that $\be(A,B) = 12 =(16)^2/(21\frac{1}{3}) > (16)^2/25$. In fact, it appears to be rather difficult to generate the maximum cut in the Clebsch graph using any local randomized procedure.

\subsection*{Other upper bounds on $\alpha_1(G)$ and $\tau_1(G)$.}
The following question is closely related to questions of Erd\H{o}s, Gallai and Tuza in~\cite{ErdProblems90,EGTTriangle,TuzaProblems}.

\begin{question}\label{q:bound2}
Determine $$c_{\tau}=\max \{ c \in \bb{R} \: : \:  \alpha_1(G)+c\tau_1(G) \leq |E(G)| \: \mathrm{for \: every \: graph}\; G \}.$$  
\end{question}

Clearly, $c_{\tau} \geq 1$. Moreover, $c_{\tau} \leq 2$, as $\alpha_1(G)+2\tau_1(G)=|E(G)|$ for the all the graphs which serve as extremal examples for Theorem~\ref{t:main}, as well as many others. To the best of our knowledge no other bounds on $c_{\tau}$ are known. Tuza (see~\cite{ErdProblems90}) conjectured that $c_{\tau} \geq 5/3$.

Following~\cite{EGTTriangle}, define $\tau_2(G)$ to be the minimum size of the set $F \subseteq E(G)$ such that every triangle in $G$ contains at least two edges of $F$. Then $\tau_2(G)= |E(G)|-\alpha_1(G)$. Therefore, $c_{\tau}$ is the maximum real number such that $\tau_2(G) \geq c_{\tau}\tau_1(G)$ for every graph $G$. 

 One can define the natural  fractional versions of $\tau_1(G)$ and $\tau_2(G)$ as follows. For $i=1,2$, let $\tau^*_i(G)$ be the minimum of $\sum_{e \in E(G)}w(e)$ taken over all functions $w:E(G) \to \bb{R}_+$ such that $\sum_{e \in E(T)}w(e) \geq i$ for every triangle $T$ in $G$. Clearly, $\tau^*_2(G) = 2\tau^*_1(G)$ for every graph $G$. Thus, perhaps, $c_{\tau}=2$, which would, in our opinion, give the most pleasing answer to Question~\ref{q:bound2}.

Unfortunately, the methods of this paper are  not applicable to Question~\ref{q:bound2}, as one can not replace  $\tau_1(G)$ by $\tau_B(G)$. (Consider for example non-bipartite triangle-free graphs.) 



\end{document}